\documentclass[11pt]{article}

\usepackage{amsmath}
\usepackage{amsfonts}
\usepackage{amssymb}
\usepackage{amsthm}
\usepackage{color}
\usepackage{bm}

\usepackage[style=numeric,giveninits=true,maxalphanames=99,maxbibnames=99]{biblatex}
\DeclareDatamodelConstant[type=list]{nameparts}{family,given,given-i}
\bibliography{GSX.bib}

\numberwithin{equation}{section}
\numberwithin{figure}{section}
\numberwithin{table}{section}

\usepackage[colorlinks]{hyperref}
\usepackage[capitalize]{cleveref}
\crefname{section}{Section}{Sections}
\Crefname{section}{Section}{Sections}

\usepackage{enumitem}
\setlist[enumerate, 1]{label = (\roman*), font = \upshape}
\setlist[itemize, 2]{label = {$\circ$}}

\usepackage[margin = 1in]{geometry}

\theoremstyle{plain}
\newtheorem{theorem}{Theorem}[section]
\newtheorem{corollary}[theorem]{Corollary}
\newtheorem{lemma}[theorem]{Lemma}

\newtheorem{proposition}[theorem]{Proposition}
\theoremstyle{definition}
\newtheorem{conjecture}[theorem]{Conjecture}

\newtheorem{example}[theorem]{Example}
\theoremstyle{remark}
\newtheorem{remark}[theorem]{Remark}
\newtheorem{claim}[theorem]{Claim}

\DeclareMathOperator{\ZZ}{\mathbb{Z}}
\DeclareMathOperator{\QQ}{\mathbb{Q}}

\DeclareMathOperator{\rank}{rank}

\DeclareMathOperator{\gt}{>}

\newcommand{\func}{{\mbox{{\sf func}}}}
\newcommand{\rk}{\mbox{\rm rank}}

\newcommand{\ve}[1]{\mathbf{#1}}
\newcommand{\cT}{\mbox{$\cal T$}}

\newcommand{\cM}{\mbox{$\cal M$}}

\newcommand{\cR}{\mbox{$\cal R$}}

\newcommand{\cV}{\mbox{$\cal V$}}
\newcommand{\cW}{\mbox{$\cal W$}}

\newcommand{\R}{\mathbb R}

\begin{document}

\title{Codes with symmetric distances}

\date{\today}

\author{
 G\'{a}bor Heged\"{u}s\thanks{Institute of Applied Mathematics, \'Obuda University}
\and
 Sho Suda\thanks{Department of Mathematics, National Defense Academy of Japan}
\and
 Ziqing Xiang\thanks{Department of Mathematics and National Center For Applied Mathematics Shenzhen, Southern University of Science and Technology}
}

\maketitle

\vspace*{-1cm}

\abstract{
For a code $C$ in a space with maximal distance $n$, we say that $C$ has symmetric distances if its distance set $S(C)$ is symmetric with respect to $n / 2$. In this paper, we prove that if $C$ is a binary code with length $2n$, constant weight $n$ and symmetric distances, then 
\[
 |C| \leq \binom{2 n - 1}{|S(C)|}.
\]
This result can be interpreted using the language of Johnson association schemes. More generally, we give a framework to study codes with symmetric distances in Q-bipartite Q-polynomial association schemes, and provide upper bounds for such codes. Moreover, we use number theoretic techniques to determine when the equality holds.
}

\section{Introduction} \label{sec:eave}
One of the central problems in coding theory is to determine the maximal number of codewords of length $n$  over $q$ symbols, subject to specific constraints such as a given minimum distance, a specified number of distinct Hamming distances (denoted by $s$, known as the {\em degree} $s$), or a constant weight $w$ for each codeword. A code with a specified degree $s$ is referred to as an {\em $s$-distance set}, and finding the best possible upper bounds for $s$-distance sets is a fundamental challenge in coding theory.

A related problem involves studying distance sets on the real unit sphere $S^{n-1}$. A common approach to establishing upper bounds on the cardinality of distance sets $C$ in various spaces $X$ is to identify a finite-dimensional vector space $V$ and a function $\varphi : X \rightarrow V$ such that the image $\varphi(C)$ forms a linearly independent set in $V$. This leads to the conclusion that $|C|\leq \dim V$. 

For $s$-distance sets $C$ in different spaces, the following upper bounds are known: 
\begin{enumerate}
\item In the Hamming space $H(n, q) := \{1,\ldots,q\}^n$ with Hamming distance $d_H(x,y)=|\{i : x_i\neq y_i, 1\leq i\leq n\}|$, the bound is: 
\[
|C|\leq \sum_{i=0}^s (q-1)^i\binom{n}{i}, 
\]
\item In the Johnson space $J(n, w) := \binom{[n]}{w}=\{x \subset \{1,\ldots,n\} : |x|=w\}$ with distance $d_J(x,y)=w-|x \cap y|$, the bound is: 
\[
|C|\leq \binom{n}{s},  
\]
\item On the unit sphere $S^{d-1}$, the bound is: 
\[
|C|\leq \binom{n-s+1}{s}-\binom{n-s+1}{s-1}. 
\]
\end{enumerate}

An $s$-distance set is called \emph{tight} if it achieves the corresponding upper bound. Examples of tight $s$-distance sets are rare and are often linked to algebraic structures, such as finite simple groups, lattices, or association schemes. The classification of tight $s$-distance sets is a compelling area of research. A well-known result in this context is the Lloyd-type theorem, which states that if a tight $s$-distance set exists in any of the aforementioned spaces, then the distances must correspond to the roots of certain orthogonal polynomials associated with each space. This theorem has been used to prove the non-existence of tight $s$-distance sets in various cases \cite{BannaiDamerell1979,BannaiDamerell1980,Hong1986}. Note that tight $s$-distance sets are tight $2s$-designs in their spaces. 

A notable class of binary codes is the \emph{self-complementary code} \cite{Levenshtein1993}. A binary code $C$ is said to be self-complementary if for every codeword $x \in C$, its complement ${\bm 1}+x$ (where ${\bm 1}$ is the all-ones vector) also belongs to $C$. The problem of $s$-distance sets has also been studied in the context of self-complementary codes, which are essentially the counterparts of $s$-distance sets in \emph{projective spaces} \cite{DelsarteGoethalsSeidel1975}, with the case $s=2$ corresponding to \emph{equiangular lines} \cite{GlazyrinYu2018,GreavesKoolenMunemasaSzoelloesi2016}.   

For a self-complementary code $C$ of length $n$, the set $S(C)$ of Hamming distances between distinct vectors in $C$ satisfies the property that if $a \in S(C) \setminus \{n\}$, then $n-a \in S(C)$. For any self-complementary code $C$, there exists a ``half'' code $C'$ such that $C = C' \cup ({\bm 1}+C')$ and $C' \cap ({\bm 1}+C') = \emptyset$. The set $C'$ satisfies the property that if $a \in S(C')$, then $n-a \in S(C')$.

\subsection{Upper bounds for codes with symmetric distances}

Recently, the first author used a linear algebraic method to prove the following theorem for codes with symmetric distances. (See also \cite[Theorem~3.2]{ArayaHaradaSuda2017} for self-complementary codes. The proof is valid for a half of a self-complementary code.)  
\begin{theorem}[\cite{Hegedues2023}] \label{thm:gh}
Let $n$ be a positive integer, and $C$ be a binary code of length $n$ with degree $s$ and degree set $S(C)=\{d_H(x,y): x,y\in C,x\neq y\}$ satisfying that $n-a\in S(C)$ if $a\in S(C)$.  
Then 
\begin{align} \label{eq:thmgh}
|C|\leq \sum_{\substack{0 \leq i \leq s\\ i\equiv s\pmod{2}}}\binom{n}{i}.
\end{align}
\end{theorem}

The first main result of this paper is an improvement of the upper bound in Theorem~\ref{thm:gh} under the assumption that $C$ is a constant weight code with length $2n$ and weight $n$. 
\begin{theorem}\label{thm:main}
Let $n$ be a positive integer, and $C$ be a binary constant weight code of length $2n$ and weight $n$ with degree $s$ and degree set $S(C)=\{d_J(x,y): x,y\in C,x\neq y\}$ satisfying that $n-a\in S(C)$ if $a\in S(C)$. 
Then 
\begin{align}\label{eq:main}
|C| \leq \binom{2n-1}{s}.
\end{align}
\end{theorem}
Note that for the codewords $x,y$ with the same weight, $d_H(x,y)=2d_J(x,y)$ holds. In \cref{sec:la}, we will give a direct proof of \cref{thm:main} using a linear algebraic method.

The constant weight codes with length $2 n$ and weight $n$ are just subsets of the Johnson association scheme $J(2 n, n)$. In the framework of the association schemes, \cref{thm:main} could be generalized. Moreover, this approach of the association schemes provides not only upper bounds on the corresponding codes, but also a necessarily condition when equality occurs. See \cref{sec:d} for undefined terms on association schemes.  
\begin{theorem}\label{thm:SCbound}
Let $(X,\{R_i\}_{i=0}^n)$ be a $Q$-bipartite $Q$-polynomial association scheme, and let $C$ be a symmetric inner distribution code with degree $s$. Then:
\begin{enumerate}
\item The following holds. 
\begin{align}\label{eq:qq}
|C|\leq \sum_{\substack{0 \leq i \leq s\\ i\equiv s\pmod{2}}}m_i.
\end{align}
\item Let
\[
  \Psi^\ast_s(x) := \sum_{\substack{0 \leq i \leq s \\ i\equiv s\pmod{2}}} v_i^*(x).
\]
If equality holds in \eqref{eq:qq}, then the zeros of $\Psi_s^*(x)$ are distinct dual eigenvalues $\theta_i^*$, $i \in S$ for some $s$-subset $S \subseteq \{1,\ldots,n-1\}$.
\end{enumerate}
\end{theorem}

\cref{thm:SCbound} will be proved at the end of \cref{sec:d}. Note that both the binary Hamming scheme $H(n, 2)$ and the Johnson association scheme $J(2 n, n)$ are Q-bipartite Q-polynomial association scheme, hence \cref{thm:SCbound} could be applied to these schemes. For the binary Hamming scheme $H(n, 2)$, $m_i = {n \choose i}$, thus \cref{thm:gh} follows from \cref{thm:SCbound} immediately. For Johnson association scheme $J(2 n, n)$, $m_i = {2 n \choose i} - {2 n \choose i - 1}$, thus \cref{thm:main} follows from \cref{thm:SCbound} immediately.

\subsection{Tight symmetric inner distribution codes}

Our second main result is on tight symmetric inner distribution codes for Johnson schemes, namely the codes for which the equality holds in \cref{thm:main}. With the help of the orthogonal polynomials for the association scheme $J(2n, n)$, we first give a strong necessary condition on when the equality holds in \cref{thm:main}.
\begin{theorem} \label{thm:jointed} Let $r, s$ be positive integers, and let $n := r + s$. Let
\begin{equation} \label{eq:pleiotaxis}
  \Phi_s(z) := \sum_{i = 0}^s (-1)^{s - i} z^{\underline{s - i}} p_{s, i} \in \QQ[z],
\end{equation}
where
\[
  p_{s, i} := \frac{{s \choose i}}{2^i} \frac{ (r + i - 1)^{\underline{\lfloor i / 2 \rfloor}} (r + 1)^{\overline{i}} }{ (r + 1 / 2)^{\overline{\lfloor i / 2 \rfloor}} }.
\]
Assume that there exist some tight symmetric inner distribution codes for the Johnson scheme $J(2 n, n)$ of degree $s$. Then, the polynomial $\Phi_s(z)$ has distinct integral zeros. Moreover, the zeros are in the interval $[1, r + s - 1]$.
\end{theorem}
The proof of \cref{thm:jointed} will be given in \cref{sec:cognizant}.

From computer experiments, we found that the zeros of $\Phi_s(z)$ are rarely all integers, unless when $s = 1$ or when $r = 1$. We thus conjecture that all tight codes arise from these two cases.

\begin{conjecture} \label{conj:madness}
Assume that there exist some tight symmetric inner distribution codes for the Johnson scheme $J(2 n, n)$ of degree $s$. Then, either $s = 1$ or $s = n - 1$.
\end{conjecture}

With the number theoretic techniques, we are able to give bounds on $s$ and $r$. The proof of \cref{thm:psychophysiology} will be given in \cref{sec:Ranzania}.
\begin{theorem} \label{thm:psychophysiology}
Let $r, s$ be positive integers. If $\Phi_s$ has only integer zeros, then either $s = 1$, or $r = 1$, or $s \geq 5000 r (14.5 + \ln r)^2$.
\end{theorem}

Note that $5000 (n - s) (14.5 + \ln (n - s))^2 \geq 2,000,000 (n - s)$. Combining \cref{thm:jointed,thm:psychophysiology}, we have a quick corollary, which verifies \cref{conj:madness} when $n \leq 4,000,000$.
\begin{corollary} \label{cor:abrachia}
Assume that there exist some tight symmetric inner distribution codes for the Johnson scheme $J(2 n, n)$ of degree $s$. Then, either $s = 1$, or $s = n - 1$, or $s > 0.9999995 n$. In particular, \cref{conj:madness} is true when $n \leq 4,000,000$.
\end{corollary}

The two cases $s = 1$ and $s = n - 1$ in \cref{conj:madness} are of special interests. We show below that the case $s = 1$ corresponds to Hadamard matrices, and the case $n - s = 1$ corresponds to maximal intersecting families.

{\bf Case} $s = 1$: Let $C$ be such a code achieving \cref{thm:main} with $s = 1$. Then, $C$ is a constant weight code with length $2 n$ and weight $n$, and $2 n - 1$ codewords satisfying $S(C)=\{ n / 2\}$. Let $H$ be an $(2 n - 1) \times 2 n$ matrix whose $i$-th row vector equals to the signed characteristic vector for the set corresponding $i$-th codeword in $C$. Then, the matrix $\begin{pmatrix} {\bf 1}^\top \\ H\end{pmatrix}$ is a Hadamard matrix of order $2 n$.

{\bf Case} $s = n - 1$: Let $C$ be such a code achieving \cref{thm:main} with $s = n - 1$. It holds that $S(C)=\{1, 2, \ldots, n - 1\}$, that is the code $C$ is an {\it intersecting family}. The upper bound is the same as the Erd\H{o}s--Ko--Rado upper bound in  \cite{ErdoesKoRado1961}. Examples $C$ attaining \eqref{thm:main} are $C=\{x\in \binom{2n}{n} : i\in x\}$ for some $i$.   

\subsection{Structure of the paper}

The organization of this paper is as follows. In Section~\ref{sec:la}, we present a direct proof of Theorem~\ref{thm:main} using linear algebraic methods. Section~\ref{sec:d} introduces the concept of $Q$-polynomial association schemes and Delsarte theory, which provide a unified framework for studying the problem of distance sets. Based on this approach, we derive an upper bound for codes with symmetric inner distribution properties in a $Q$-bipartite $Q$-polynomial association scheme, and prove \cref{thm:SCbound}. \cref{sec:cognizant} focuses on an explicit description of the polynomial $\Psi^\ast_s$ used in \cref{thm:SCbound}. We prove that $\Psi^\ast_s$ is essentially $\Phi_s$ in \cref{eq:pleiotaxis}, and prove \cref{thm:jointed}. \cref{sec:Ranzania} studies the zeros of $\Phi_s$, and prove \cref{thm:psychophysiology}.

\subsection{Notation}

Throughout the paper, we identify a subset $F$ in $\{1,\ldots,2n\}$ with its characteristic vector, that is a $(0,1)$-vector indexed by the elements of $\{1,\ldots,2n\}$ with $i$-th entry $1$ if $i\in F$ and $0$ if $i\not\in F$.  
For a positive integer $n$, let $[n]=\{1,\ldots,n\}$ and $\binom{[n]}{w}$ be the set of all subset of size $w$ in $[n]$. 

\section{A linear algebraic method}\label{sec:la}

\subsection{Functions on the finite sets}

Let $\cT \subseteq \{-1,1\}^{2n}$ be a fixed subset. Denote by 
$\func(\cT,\R):=\{f: \cT\to \R: f\mbox{ is a function}\}$ the set  of all functions from the set $\cT$ to the field $\R$. 

For $k\in\{s-1,s\}$, let
\begin{align*}
\cR(2n,k)&:=\left\{\alpha=(\alpha_1, \ldots , \alpha_{2n})\in \{0,1\}^{2n}: \sum_{i=1}^{2n} \alpha_i \equiv k \pmod 2,\ \sum_{i=1}^{2n}\alpha_i\leq k\right\}, 
\end{align*}
 and let $S(2n,k)$ denote the subspace of the vector space $\func(\{-1,1\}^{2n},\R)$ generated by the set of monomials
$
\{x^{\alpha}: \alpha\in \cR(2n,k)\}.  
$
Then it follows easily from \cite[Theorem 7.8]{BabaiFrankl2022} and the definition of $\cR(2n,k)$ that 
\begin{align}\label{eq:dims}
\dim(S(2n,k))=\sum_{\substack{i=0,1,\ldots, k\\ i\equiv k\pmod{2}}}\binom{2n}{i}.
\end{align}

A \emph{signed characteristic vector} of a subset $F$ in $[2n]$ is a function $\ve v$ from the set $[2n]$ to the real number field $\mathbb{R}$ such that $\ve v(k)=1$ for $k\in F$, $\ve v(k)=-1$ otherwise. 

For a polynomial $f=f(x_1,\ldots,x_{2n})$,  let $\overline{f}$ denote the polynomial obtained from $f$ by replacing each  occurrence of $x_j^2$, $1\leq j\leq 2n$,  by $1$.
Note that $f(\ve v)=\overline{f}(\ve v)$ for any $(1,-1)$-vector $\ve v$.  

Let $\cV_{< n}:=\{\ve u\in\{-1,1\}^{2n}: |\{i: u_i=1\}|< n \}$ where $\ve u=(u_1,\ldots,u_{2n})$. 
Let $t:=|\cV_{<n}|$. 
We can write up $\cV_{< n}$ as the set of vectors $\cV_{< n}=\{\ve w_1, \ldots ,\ve w_t\}$.  Define the polynomial $h(\ve x):=\sum_i x_i$. Clearly $h(\ve w_j)\neq 0$ for each $1\leq j\leq t$.            
Let $\cW:=\{\ve u\in\{-1,1\}^{2n}: |\{i: u_i=1\}|=n \}$. 
Note that the signed characteristic vectors of $C$ belong to $\cW$ and $h(\ve w)=0$ for each $\ve w\in \cW$.

For $\alpha\in \cR(2n,s-1)$, let $g_\alpha:=x^\alpha \cdot h$. 
Clearly $g_\alpha$ is a homogeneous polynomial and $\deg(g_\alpha)\equiv s \pmod 2$, hence $\overline{g}_\alpha\in S(2n,s)$.

\begin{proposition} \label{indep}
Let 
$$
\cM:=\{\overline{g}_\alpha : \alpha\in \cR(2n,s-1)\}.
$$
Then $\cM$ is a linearly independent set of functions in the vector space $\func(\{-1,1\}^{2n},\R)$. 
\end{proposition}
\begin{proof}
We can write up $S(2n,s-1)$ as the set of monomials $S(2n,s-1)=\{x^{\alpha_1}, \ldots ,x^{\alpha_z}\}$, where $z:=|S(2n,s-1)|$.

Let $A,B$ be the $t\times z$ matrices defined by 
$$
A(i,j):=x^{\alpha_j}(\ve w_i),\quad B(i,j):= (h\cdot x^{\alpha_j})(\ve w_i),
$$
where $i\in [t]$, $j\in [z]$. Then by \cref{eq:dims}, $\rk(A)=z$. 

Since $h(\ve w_i)\neq 0$ for each $1\leq i\leq t$, hence each row of $B$ is a nonzero multiple of the corresponding row of $A$. We get that $\rk(B)=z$. This means precisely that the set of polynomials
$$
\{ g_\alpha : \alpha\in \cR(2n,s-1)\}
$$
is a linearly independent set of functions in the vector space $\func(\cV_{<n},\R)$.

Hence the set of polynomials $\cM$ is a linearly independent set of functions in the vector space $\func(\cV_{<n},\R)$, consequently in the vector space  $\func(\{-1,1\}^{2n},\R)$. 
\qed

\subsection{Proof of Theorem~\ref{thm:main}}
 Let $C=\{F_1, \ldots ,F_m\}$
  be a subset of $\binom{[2n]}{n}$ with degree $s$ and degree set $S(C)=\{n-|F\cap F'|: F,F'\in C,F\neq F'\}$ satisfying that $n-a\in S(C)$ if $a\in S(C)$. 
Let $\ve v_i$ denote the signed characteristic vector of the subset $F_i\in C$ in $[2n]$.
Consider the polynomials
\begin{equation}  \label{Ppol} 
P_i(x_1, \ldots ,x_{2n}):= \prod_{d\in S(C)} \Big( \langle \ve x, \ve v_i\rangle-2(n-2d) \Big)\in \R[\ve x]
\end{equation}
for each $1\leq i\leq m$, where $\langle \ve x, \ve v_i\rangle$ denotes the usual scalar product of the vectors $\ve x$ and $ \ve v_i$. 
Clearly $\mbox{deg}(P_i)\leq s$ for each $1\leq i\leq m$. 
Since $S(C)$ is symmetric with respect to $n/2$, the set $\{2(n-2d) : d\in S(C)\}$ is symmetric with respect to $0$.  
Then it follows that each $\overline{P}_i\in S(2n,s)$. 

\begin{claim}\label{cl:1}
The set of polynomials $\{\overline{P}_i: 1\leq i\leq m\}\cup \cM$ is a linearly independent set of functions in the vector space $\func(\{-1,1\}^n,\R)$.
\end{claim}
\begin{proof}[Proof of Claim~\ref{cl:1}]
Assume that
$$
\sum_{i=1}^m \lambda_i\overline{P}_i+ \sum_{\alpha\in \cR(2n,s-1)} \mu_\alpha h\cdot x^\alpha=0
$$
for some $\lambda_i, \mu_\alpha\in \R$. Substituting $\ve v_i$, all terms in the second sum vanish, because $h(\ve w)=0$ for each $\ve w\in \cW$, and only the term with subscript $i$ remains of the first sum. 
We infer that $\lambda_i=0$ for each $i$ because $\overline{P}_i(\ve v_j)$ is non-zero if $i=j$ and zero otherwise. 
Therefore the previous linear relation is only a relation among the polynomials $h\cdot x^\alpha$. Then the claim follows from  Proposition~\ref{indep}.
\end{proof}

We thus found $m+|\cM|$ linearly independent functions, all of which is an element of the set of polynomials $S(2n,s)$, 
hence by \cref{eq:dims} with $|\cM|=\dim(S(2n,s-1))$ and the identity \cite[(5.16)]{GrahamKnuthPatashnik1994}, 
$$
m\leq \dim(S(2n,s))-\dim(S(2n,s-1))=\sum_{\substack{0 \leq i \leq s\\ i\equiv s\pmod{2}}}\binom{2n}{i}- \sum_{\substack{0 \leq i \leq s-1\\ i\equiv s-1\pmod{2}}}\binom{2n}{i}= \binom{2n-1}{s},
$$
which completes the proof of Theorem~\ref{thm:main}. 
\end{proof}

\section{Q-bipartite Q-polynomial association scheme approach} \label{sec:d}
\subsection{Association schemes}
We refer the reader to \cite{BannaiIto1984} for more informations. 
Let $X$ be a finite set and $\{R_0,R_1,\ldots,R_n\}$ be a set of non-empty subsets of $X\times X$.
Let $A_i$ be the adjacency matrix of the graph $(X,R_i)$.
The pair $(X,\{R_i\}_{i=0}^n)$ is a symmetric association scheme with class $n$ if the following hold:
\begin{enumerate}
\item $A_0 $ is the identity matrix;
\item $\sum_{i=0}^nA_i=J$, where $J$ is the all ones matrix;
\item $A_i^\top=A_i$ for $1\leq i\leq n$;
\item $A_iA_j$ is a linear combination of $A_0,A_1,\ldots,A_n$ for $0\leq i,j\leq n$.
\end{enumerate}
The vector space $\mathcal{A}$ spanned by the $A_i$ over $\mathbb{R}$ forms an algebra which is called the adjacency algebra of $(X,(X,\{R_i\}_{i=0}^n))$.
Since $\mathcal{A}$ is commutative and semisimple, there exist primitive idempotents $\{E_0,E_1,\ldots,E_n\}$ where $E_0=\frac{1}{|X|}J$.
Since the adjacency algebra $\mathcal{A}$ is closed under the ordinary multiplication and entry-wise multiplication denoted by $\circ$, 
we define the intersection numbers $p_{i,j}^k$ and the Krein numbers $q_{i,j}^k$ for $0\leq i,j,k\leq d$ as follows;
\begin{align*}
A_iA_j=\sum_{k=0}^n p_{i,j}^kA_k,\quad E_i\circ E_j=\frac{1}{|X|}\sum_{k=0}^n q_{i,j}^kE_k.
\end{align*}

Since $\{A_0,A_1,\ldots,A_n\}$ and $\{E_0,E_1,\ldots,E_n\}$ form bases of $\mathcal{A}$, there exist change of bases matrices $P$ and $Q$ defined by
\begin{align*}
A_i=\sum_{j=0}^n P_{ji}E_j,\quad E_i=\frac{1}{|X|}\sum_{j=0}^n Q_{ji}A_j.
\end{align*}
The matrices $P$ and $Q$ are called the first and second eigenmatrices of $(X,\{R_i\}_{i=0}^n)$ respectively.
The multiplicity $m_i$ is defined to be $m_i=Q_{0i}=\text{rank}E_i$.  

The symmetric association scheme $(X,\{R_i\}_{i=0}^n)$ is said to be $Q$-polynomial 
if for each $i\in\{0,1,\ldots,n\}$, there exists a polynomial $v_i^*(x)$ with degree $i$ such that $Q_{ji}=v_i^*(Q_{j1})$ for all $j\in\{0,1,\ldots,n\}$.
If $(X,\{R_i\}_{i=0}^n)$ is $Q$-polynomial, we define $a_i^*=q_{1,i}^i$, $b_i^*=q_{1,i+1}^i$, and $c_i^*=q_{1,i-1}^i$. 
Then $\{v_i^*(x)\}_{i=0}^n$ satisfy that $v_{0}^*(x)=1$, $v_1^*(x)=x$ and $x v_i^*(x)=c_{i+1}^*v_{i+1}^*(x)+a_{i}^*v_{i}^*(x)+b_{i-1}^*v_{i-1}^*(x)$ for $1\leq i \leq n$, where $v_{n+1}^*(x)$ is indeterminate. 
When $(X,\{R_i\}_{i=0}^n)$ is $Q$-polynomial, set $\theta_j^*=Q_{j1}$ for each $j\in\{0,1,\ldots,n\}$.   
 $\theta_j^*$ are called \emph{dual-eigenvalues} of the $Q$-polynomial scheme. 
The $Q$-polynomial scheme is said to be {\it $Q$-bipartite} if $a_i^*=0$ for $i=1,\ldots,n$. 
Then it holds that $\theta^*_{n-i}=-\theta^*_i$ for any $i$. 
Define 
\[
\Psi_s^*(z)=\sum_{\substack{0 \leq i \leq s\\ i\equiv s\pmod{2}}}v_i^*(z). 
\]
\begin{example}
The following are examples of $Q$-bipartite $Q$-polynomial schemes:
\begin{itemize}
\item The Hamming association scheme $H(n,2)$ is a pair $(X,\{R_i\}_{i=0}^n)$ where $X=\{0,1\}^n$ and $R_i=\{(x,y)\in X\times X : d_H(x, y)=i\}$ for $i\in\{0,1,\ldots,n\}$, where $d_H(x,y)$ is the Hamming distance between $x$ and $y$. 
The scheme is a $Q$-bipartite $Q$-polynomial association scheme \cite[Section III.2]{BannaiIto1984}. 
\item The Johnson association scheme $J(2n,n)$ is a pair $(X,\{R_i\}_{i=0}^n)$ where $X=\binom{[2n]}{n}$ is the set of all subset of size $n$ in $[2n]$ and $R_i=\{(x,y)\in X\times X : d_J(x,y)=i\}$ for $i\in\{0,1,\ldots,n\}$. 
The scheme is a $Q$-bipartite $Q$-polynomial association scheme \cite[Section III.2]{BannaiIto1984}. 
\end{itemize} 
\end{example}

\subsection{Symmetric inner distribution codes in $Q$-bipartite $Q$-polynomial schemes}
Let $(X,\{R_i\}_{i=0}^n)$ be a $Q$-polynomial association scheme. 
Let $C$ be a subset of $X$, called a \emph{code}. 
For a non-empty subset $C$ in $X$, we define the characteristic vector $\chi$ as a column vector indexed by $X$ whose $x$-th entry is $1$ if $x\in C$, and $0$ otherwise. 

We define the degree set $S(C)$ of $C$ as 
\begin{align*}
S(C)=\{i : 1\leq i \leq n, \chi^\top A_i\chi\neq0 \},
\end{align*}
and also define degree $s$ 
as the cardinality of the degree set $S(C)$. 
When the scheme is $Q$-polynomial, a polynomial $F(x)$ is called an \emph{annihilator polynomial} of $C$ if 
\begin{align*}
F(\theta_i^*)=0\ \text{ for any } i\in S(C).
\end{align*}
Since $E_i$ is positive semidefinite, there exits a $|X|\times m_i$ real matrix $H_i$ such that $E_i=\frac{1}{|X|}H_i H_i^\top$. 
The $i$-th characteristic matrix $G_i$ fof $C$ is a submatrix of $H_i$ obtained by restricting the rows to the elements of $C$.

Let $(X,\{R_i\}_{i=0}^n)$ be a $Q$-bipartite $Q$-polynomial association scheme. 
A code $C$ is said to be a {\it symmetric inner distribution code} if $a\in S(C)$ implies  $n-a\in S(C)$.

\begin{lemma}\label{lem:eq}
Let $S$ be a subset of $\{1,2,\ldots,n-1\}$ such that $|S|=s$, 
and if $a\in S $ then $n-a\in S$.
Then 
$F(z)=\prod_{i\in S}(z-\theta_i^*)$ 
has the following expansion by the orthogonal polynomials $v_i^*(z)$:  
\begin{align}\label{eq:alpha}
F(z)
=\sum_{\substack{0 \leq i \leq s\\ i\equiv s\pmod{2}}}f_i v_i^*(z),
\end{align}
where $f_i \in \mathbb{R}$.
\end{lemma}
\begin{proof}
When $s$ is even, we may write
 $S=\{a_1,a_2,\ldots,a_{s/2},n-a_1,n-a_2,
\ldots,n-a_{s/2}\}$, where
 $0<a_i <n/2$ for $1\leq i\leq s/2$. 
Then by the fact that $\theta^*_{n-i}=-\theta^*_i$, we have
\begin{align*}
F(z)
=\prod_{i=1}^{s/2}
\left(z-\theta^*_{a_i}\right)\left(z-\theta^*_{n-a_{i}}\right)
=\prod_{i=1}^{s/2}
\left(z-\theta^*_{a_i}\right)\left(z+\theta^*_{a_i}\right)
=\prod_{i=1}^{s/2}
\left(z^2-(\theta^*_{a_i})^2\right).
\end{align*}
Thus, $F(z)$ is 
an even polynomial in variable $z$.

When $s$ is odd, we may write $S=\{a_1,a_2,\ldots,a_{(s-1)/2},n/2,
n-a_1,n-a_2,\ldots,n-a_{(s-1)/2}\}$, where $0<a_i <n/2$ for $1\leq i\leq (s-1)/2$ and $n$ must be even. 
Similar to the case where $s$ is even, we have 
\begin{align*}
F(z)
&=z\prod_{i=1}^{(s-1)/2}\left(z^2-(\theta^*_{a_i})^2\right). 
\end{align*}
Thus, $F(z)$ is 
an odd polynomial in variable $x$.

It can be shown that $v_i^*(z)$ is an even (resp.\ odd) 
polynomial of degree $i$ in variable $x$ if $i$ is an even (resp.\ odd),  
from which  the expansion of $F(z)$ by the 
orthogonal polynomials has the desired form~\eqref{eq:alpha}.
\end{proof}

Now, we are ready to prove \cref{thm:SCbound}.

\begin{proof}[Proof of \cref{thm:SCbound}]
(i): The annihilator polynomial of $C$ is defined as follows:
\begin{align*}
F(z)&=|C|\prod_{i\in S(C)}\frac{z-\theta_i^*}{\theta_0^*-\theta_i^*}.
\end{align*}
Since $C$ is a symmetric inner distribution code, 
the annihilator polynomial $F(z)$ of $C$ 
has the following expansion by Lemma~\ref{lem:eq}: 
\begin{align}\label{eq:coef}
F(z)
=\sum_{\substack{0 \leq i \leq s\\ i\equiv s\pmod{2}}}f_i v_i^*(z).
\end{align}

Set $K=\begin{pmatrix}G_0&G_2&\cdots &G_s \end{pmatrix}$ 
if $s$ is even and  
$K=\begin{pmatrix}G_1&G_3&\cdots &G_s \end{pmatrix}$
if $s$ is odd, 
where $G_i$ is the $i$-th characteristic matrix of $C$,
and set
\begin{align*}
\Gamma=\bigoplus_{\substack{0 \leq i \leq s\\ i\equiv s\pmod{2}}}f_i I_{m_i},
\end{align*} 
where $I_m$ is the identity matrix of order $m$. 
By \cite[Theorem~3.13]{Delsarte1973}, we have 
$K\Gamma K^\top=|C|I_{|C|}$. 
Taking the rank of the equation yields that 
\begin{align*}
|C|=\rank(K\Gamma K^\top)\leq\rank(K)\leq 
\sum_{\substack{0 \leq i \leq s\\ i\equiv
s\pmod{2}}}m_i, 
\end{align*}
as desired. 

(ii): 
We investigate when equality hold. 
Assume equality holds in Theorem~\ref{thm:main}. 
It follows from \cite[Lemma~3.6]{Delsarte1973} that $f_i$ in \eqref{eq:coef} for all $i$ satisfy $f_i\leq 1$. 
Since 
\[
|C|=F(\theta^\ast_0)=\sum_{\substack{0 \leq i \leq s\\ i\equiv s\pmod{2}}}f_i v_i^*(\theta^\ast_0)=\sum_{\substack{0 \leq i \leq s\\ i\equiv s\pmod{2}}}f_i m_i\leq \sum_{\substack{0 \leq i \leq s\\ i\equiv s\pmod{2}}}m_i,
\]
$f_i=1$ must hold for all $i\equiv s \pmod{2}$. 
Therefore its annihilator polynomial satisfies that 
$F(\theta^\ast_z)=\Psi_s^*(z)$. 
Since the polynomial $F(\theta^\ast_z)$ is an annihilator polynomial of $C$, the result follows. 
\end{proof}

\section{Johnson scheme $J(2n, n)$} \label{sec:cognizant}

We focus on the Johnson scheme $J(2n, n)$ in this section. The goal is to provide an explicit formula for the polynomial $\Psi^\ast_s(x)$ for Johnson scheme $J(2n, n)$. The dual orthogonal polynomials for $J(2 n, n)$ are $v^\ast_i(x) \in \QQ(n)[x]$, defined by the three-term recursion
\begin{equation} \label{eq:slee}
  x v^\ast_i(x) = \frac{(2 n - 1) (i + 1) (n - i)}{n (2 n - 2 i - 1)} v^\ast_{i + 1}(x) + \frac{(2 n - 1) (n - i + 1) (2 n - i + 2)}{n (2 n - 2 i + 3)} v^\ast_{i - 1}(x)
\end{equation}
with $v^\ast_{-1}(x) = 0$ and $v^\ast_0(x) = 1$.

The \cref{prop:amphivorous} below proves that $\Psi^\ast_s(x)$ is essentially the same with the polynomial $\Phi_s(z)$ defined in \cref{eq:pleiotaxis}, up to some scalar and change of variable. \cref{thm:jointed} is a direct corollary of \cref{prop:amphivorous,thm:SCbound}, and the proof of \cref{thm:jointed} is given at the end of this section.

\begin{proposition} \label{prop:amphivorous}
Let $\theta_z := (2 n - 1)(1 - 2 z / n) \in \QQ(n)[z]$. Then,
\[
   \Psi^\ast_s(\theta_z) = \frac{2^{2 s} }{s!} \frac{(n - 1 / 2)^{\underline{s}} }{n^{\underline{s}}} \Phi_s(z).
\]
\end{proposition}

\begin{proof}
Substitute $x = \theta_z$ into the recursion \cref{eq:slee} for $v^\ast_i$, we obtain
\begin{equation} \label{eq:Betty}
  (n - 2 z) v^\ast_i(\theta_z) = \frac{(i + 1) (n - i)}{2 n - 2 i - 1} v^\ast_{i + 1}(\theta_z) + \frac{(n - i + 1) (2 n - i + 2)}{2 n - 2 i + 3} v^\ast_{i - 1}(\theta_z),
\end{equation}
with $v^\ast_{-1}(\theta_z) = 0$ and $v^\ast_0(\theta_z) = 1$. Thus, the polynomials $v^\ast_i(\theta_z) \in \QQ(n)[z]$ are uniquely determined by the recursion \cref{eq:Betty} and the initial conditions.

Let
\[
  \Phi'_s(z) := \frac{2^{2 s} }{s!} \frac{(n - 1 / 2)^{\underline{s}} }{n^{\underline{s}}} \Phi_s(z) = \sum_{i = 0}^s \alpha_{s, i} (-1)^{s - i} z^{\underline{s - i}} \in \QQ(n)[z],
\]
where
\[
  \alpha_{s, i} := \frac{2^{2 s - i} {s \choose i}}{s!} \frac{(n - 1 / 2)^{\underline{s - \lfloor i / 2 \rfloor}} (n - s + i - 1)^{\underline{\lfloor i / 2 \rfloor}} }{n^{\underline{s - i}}} \in \QQ(n).
\]

Note that $v^\ast_s(\theta_z) = \Psi^\ast_s(\theta_z) - \Psi^\ast_{s - 2}(\theta_z)$. In order to prove that $\Psi^\ast_s(\theta_z) = \Phi'_s(z)$, it suffices to prove that $\Phi'_s(z) - \Phi'_{s - 2}(z)$ also satisfies the recursion \cref{eq:Betty} for $v^\ast_s(\theta_z)$. In other words, if suffices to prove that, for each $s$,
\begin{equation} \label{eq:noteless}
\begin{aligned}
& (n - 2 z) (\Phi'_s(z) - \Phi'_{s - 2}(z)) \\
= & \frac{(s + 1) (n - s)}{2 n - 2 s - 1} (\Phi'_{s + 1}(z) - \Phi'_{s - 1}(z)) + \frac{(n - s + 1) (2 n - s + 2)}{2 n - 2 s + 3} (\Phi'_{s - 1}(z) - \Phi'_{s - 3}(z))
\end{aligned}
\end{equation}
as elements in $\QQ(n)[z]$.

The set $\{(-1)^t z^{\underline{t}} : t \in \mathbb{N}\}$ is a basis of $\QQ[z]$, and
\[
  (n - 2 z) \left( (-1)^t z^{\underline{t}} \right) = 2 \left( (-1)^{t + 1} z^{\underline{t + 1}} \right) + (n - 2 t) \left( (-1)^t z^{\underline{t}} \right).
\]
By taking the coefficient of $(-1)^t z^{\underline{t}}$ in \cref{eq:noteless} and then divide by $\alpha_{s, s - t}$, we see that \cref{eq:noteless} holds for all $s$, if and only if, for all $s$ and $t$,
\begin{equation} \label{eq:unprecipitate}
\begin{aligned}
& 2 (\frac{\alpha_{s, s - t + 1}}{\alpha_{s, s - t}} - \frac{\alpha_{s - 2, s - t - 1}}{\alpha_{s, s - t}}) + (n - 2 t) (1 - \frac{\alpha_{s - 2, s - t - 2}}{\alpha_{s, s - t}}) \\
= & \frac{(s + 1) (n - s)}{2 n - 2 s - 1} (\frac{\alpha_{s + 1, s - t + 1}}{\alpha_{s, s - t}} - \frac{\alpha_{s - 1, s - t - 1}}{\alpha_{s, s - t}}) \\
& + \frac{(n - s + 1) (2 n - s + 2)}{2 n - 2 s + 3} (\frac{\alpha_{s - 1, s - t - 1}}{\alpha_{s, s - t}} - \frac{\alpha_{s - 3, s - t - 3}}{\alpha_{s, s - t}})
\end{aligned}
\end{equation}
as elements in $\QQ(n)$.

For $s$ and $t$ such that $s - t$ is even, every $\alpha$ quotient in \cref{eq:unprecipitate} is in $\QQ(n, s, t)$, thus both sides of \cref{eq:unprecipitate} are actually rational functions in $\QQ(n, s, t)$. Similar is true when $s - t$ is odd. Therefore, by simply expanding both sides, \cref{eq:unprecipitate} is verified. Thus, \cref{eq:noteless} is true, and the result follows.
\end{proof}

Now, we combine \cref{thm:SCbound,prop:amphivorous} to prove \cref{thm:jointed}.

\begin{proof}[Proof of \cref{thm:jointed}]
By \cref{thm:SCbound}, the zeros of $\Psi^\ast_s$ are distinct dual eigenvalues $\theta^\ast_i$, $i \in S$ for some $s$-subset $S \subseteq \{1, \ldots, n - 1\}$. Thus, by \cref{prop:amphivorous}, The zeros of $\Phi_s$ are distinct integers $i$, $i \in S$.
\end{proof}

\section{Zeros of $\Phi_s(z)$} \label{sec:Ranzania}

The goal of this section is to determine when the zeros of $\Phi_s(z)$ are all integers. In \cref{sec:rokeage}, we show a technical result on intervals with no primes proved in \cite{Xiang2023}, which will be used in latter part of the section. In \cref{sec:wheezer}, we develop some technical results on binomial coefficients with half-integers. In \cref{sec:canescence}, we analyze the coefficients of $\Phi_s(z)$, namely the $p_{s, i}$'s. With the help of results in \cref{sec:rokeage,sec:wheezer}, we prove \cref{thm:psychophysiology}.

\subsection{Intervals with no primes} \label{sec:rokeage}

Let $\rho_s$ be the smallest positive integer $n$ such that there are no primes in the interval $(n, n + s - 1]_{\ZZ}$. Using the known estimates on prime gaps, \cite{Xiang2023} gives an estimate of $\rho_s$. For reader's convenience, we show the result here.

\begin{proposition}[{\cite[Proposition 5.3]{Xiang2023}}] \label{prop:dominated}
For $s \geq 288$, $\rho_{s} \geq 5000 s (14.6 + \ln s)^2 > 2,000,000 s$.
\end{proposition}

\subsection{Binomial coefficients for half-integers} \label{sec:wheezer}

In this subsection, we study the binomial coefficients for half-intergers
\begin{equation} \label{eq:schistothorax}
  {n - 1 / 2 \choose m} = \frac{(2 n - 1)!!}{(2 (n - m) - 1)!! 2^m m!}.
\end{equation}
The $n$ in this subsection is just a natural number, which is not the parameter $n$ for the Johnson scheme $J(2 n, n)$.

For a prime $p$, let ${\rm val}_p$ be the $p$-adic valuation for rational numbers. In other words, it is the number of times that $p$ divides a given rational number. Below is a generlaization of Kummer's theorem to half-integers.

\begin{lemma} \label{lem:trimellitic}
Let $m$ and $n$ be natural numbers with $n \geq m$. Let $p$ be a prime. Then,
\[
  {\rm val}_p({n - 1 / 2 \choose m}) \leq \begin{cases}
    \lfloor \log_p (2 n - 1) \rfloor, & p \geq 3, \\
    - 2 m + \lfloor \log_2 (m + 1) \rfloor, & p = 2. \\
  \end{cases}
\]
\end{lemma}

\begin{proof}
When $p = 2$, the result follows from \cref{eq:schistothorax} and the fact that
\[
{\rm val}_2( m! ) \geq m - \lfloor \log_2 (m + 1) \rfloor.
\]
Now, assume that $p \geq 3$.

The Legendre's formula for factorials says
\begin{equation} \label{eq:semicorneous}
  {\rm val}_p (n!) = \sum_{i = 1}^\infty \left\lfloor \frac{n}{p^i} \right\rfloor.
\end{equation}
Following the idea of its proof, we obtain a generalization for double factorials as follows:
\begin{equation} \label{eq:Ashkenazic}
  {\rm val}_p ((2 n - 1)!!) = \sum_{i = 1}^\infty \left\lceil \frac{1}{2} \left\lfloor \frac{2 n - 1}{p^i} \right\rfloor \right\rceil.
\end{equation}
Substituting \cref{eq:semicorneous,eq:Ashkenazic} into \cref{eq:schistothorax},
\begin{align*}
{\rm val}_p({n - 1 / 2 \choose m})
= & {\rm val}_p( (2 n - 1)!! ) - {\rm val}_p( (2 (n - m) - 1)!! ) - {\rm val}_p( m! ) \\
= & \sum_{i = 1}^\infty \left( \left\lceil \frac{1}{2} \left\lfloor \frac{2 n - 1}{p^i} \right\rfloor \right\rceil - \left\lceil \frac{1}{2} \left\lfloor \frac{2 (n - m) - 1}{p^i} \right\rfloor \right\rceil - \left\lfloor \frac{m}{p^i} \right\rfloor \right), \\
= & \sum_{i = 1}^\infty \left( \left\lceil \frac{1}{2} \left\lfloor a_i + b_i \right\rfloor \right\rceil - \left\lceil \frac{1}{2} \left\lfloor a_i \right\rfloor \right\rceil - \left\lfloor \frac{1}{2} b_i \right\rfloor \right),
\end{align*}
where
\[
a_i := \frac{2 (n - m) - 1}{p^i} \quad \text{and} \quad b_i := \frac{2 m}{p^i}.
\]
The result follows from the fact that each summand is either $0$ or $1$, and that the summand is $0$ unless $a_i + b_i \geq 1$. 
\end{proof}

Now, we analyze when the binomial coefficients for half-integers have only small prime factors. We say that a rational number $n$ is {\em $b$-smooth} if ${\rm val}_p(n) \leq 0$ for all primes $p \gt b$.

\begin{proposition} \label{prop:zinziberaceous}
Let $n$ and $m$ be natural numbers with $n \geq m$. If ${n - 1 / 2 \choose m}$ is $(4 m + 1)$-smooth, then either $m \leq 142$, or $n \leq 3 m - 1$.
\end{proposition}

\begin{proof}
Assume in contrary that $m \geq 143$ and $n \geq 3 m$.

We first give an upper bound for $\ln {n - 1 / 2 \choose m}$ using its $(4 m + 1)$-smoothness property. Note that the $p$ in the calculation below runs over prime numbers, and $\pi$ is the prime counting function.
\begin{align*}
& \ln {n - 1 / 2 \choose m} = \sum_{p \leq 4 m + 1} {\rm val}_p({n - 1 / 2 \choose m}) \ln p \\
\leq & \left(- 2 m + \lfloor \log_2 ( m + 1 ) \rfloor\right) \ln 2 + \sum_{3 \leq p \leq 4 m + 1} \left\lfloor \log_p (2 n - 1) \right\rfloor \ln p & \text{(by \cref{lem:trimellitic})} \\
\leq & - 2 m \ln 2 + \ln (m + 1) + (\pi(4 m + 1) - 1) \ln (2 n - 1) \\
\leq & - 2 m \ln 2 + \pi(4 m + 1) \ln (2 n - 1) & \text{(since $2 n - 1 \geq m + 1$)} \\
\leq & \pi(4 m + 1) \ln (n - 1 / 2). & \text{(since $\pi(4 m + 1) \leq 2 m$)}
\end{align*}

We then give a lower bound for $\ln {n - 1 / 2 \choose m}$ using Stirling's formula. Recall that Stirling's formula says
\begin{equation} \label{eq:ethylenic}
  (n + 1 / 2) \ln n - n < \ln n! \leq (n + 1 / 2) \ln n - n + 1.
\end{equation}
for all real $n \geq 1$. We proceed with the following calculation.
\begin{align*}
& \ln {n - 1 / 2 \choose m} \\
\geq & \left(n \ln (n - 1 / 2) - (n - 1 / 2) \right) \\
& - \left((n - m) \ln (n - m - 1 / 2) - (n - m - 1 / 2) + 1\right) & \text{(by \cref{eq:schistothorax,eq:ethylenic})} \\
& - \left((m + 1 / 2) \ln m - m + 1\right) \\
= & n \ln (n - 1 / 2) - (n - m) \ln (n - m - 1 / 2) - (m + 1 / 2) \ln m - 2 \\
= & m (\ln (n - 1 / 2) - \ln m) + (n - m) \ln (1 + \frac{m}{n - m - 1 / 2}) - \frac{1}{2} \ln m - 2\\
\geq & m (\ln (n - 1 / 2) - \ln m) + 2 m \ln \frac{3}{2} - \frac{1}{2} \ln m - 2 & \text{(since $n \geq 3 m$)}\\
\geq & m (\ln (n - 1 / 2) - \ln m + \ln 2). & \text{(since $m \geq 32$)}
\end{align*}

Combining the upper bound and the lower bound for $\ln {n - 1 / 2 \choose m}$,
\begin{align*}
  & \ln (n - 1 / 2) \\
  \leq & \left(1 - \frac{1}{m} \pi(4 m + 1)\right)^{-1} (\ln m - \ln 2) & \text{(since $\pi(4 m + 1) < m$ for $m \geq 31$)} \\
  \leq & \left(1 - \frac{1}{m} \frac{4 m + 1}{\ln (4 m + 1)} (1 + \frac{1.2762}{\ln (4 m + 1)})\right)^{-1} (\ln m - \ln 2) & \text{(estimate of $\pi$ in \cite[Corollary 5.2]{Dusart2018})} \\
  \leq & \ln 2,000,000 m. & \text{(since $m \geq 116$)}
\end{align*}
In other words, $n \leq 2,000,000 m$.

On the other hand,
\[
  {n - 1 / 2 \choose m} = \frac{(2 n - 2 m + 1) (2 n - 2 m + 3) \ldots (2 n - 1)}{(2 m)!!}.
\]
Since it is $(4 m - 1)$-smooth and $2 n - 2 m + 1 \geq 4 m + 1$, neither of $2 n - 2 m + 1, \dots, 2 n - 1$ could be a prime. Therefore, since $m \geq 143$, by \cref{prop:dominated},
\[
  2 n - 2 m - 1 \geq \rho_{2 m + 2} \gt 4,000,000 m.
\]
Thus, $n > 2,000,000 m$, which contradict with $n \leq 2,000,000 m$ above.
\end{proof}

\begin{remark}
For the cases $31 \leq m \leq 142$, it is possible to improve \cref{prop:zinziberaceous} by using better estimates. But, \cref{prop:zinziberaceous} is already good enough for our purposes.
\end{remark}

\subsection{Proof of \cref{thm:psychophysiology}} \label{sec:canescence}

We begin with a technical result.

\begin{proposition} \label{prop:unensured}
Let $s, r$ be positive integers numbers with $s \geq 2$. Suppose that $p_{s, 2}$ is an integer. Then, $r < \frac{3}{4} s^2$.
\end{proposition}

\begin{proof}
The $p_{s, 2}$, treated as a rational function in $r$, admits an expansion
\[
  16 p_{s, 2} = {s \choose 2} \left(4 r^2 + 14 r + 13 + \frac{3}{2 r + 1}\right),
\]
from which the result follows.
\end{proof}

Now, we combine the results in \cref{sec:wheezer} to prove \cref{thm:psychophysiology}.

\begin{proof}[Proof of \cref{thm:psychophysiology}]
Assume in contrary that $s \geq 2$ and $r \geq 2$. Recall that
\[
  p_{s, s} := \frac{ (r + s - 1)^{\underline{\lfloor s / 2 \rfloor}} (r + 1)^{\overline{s}} }{ 2^s (r + 1 / 2)^{\overline{\lfloor s / 2 \rfloor}} }.
\]
Let
\[
f := (x + s - 1)^{\underline{\lfloor s / 2 \rfloor}} (x + 1)^{\overline{s}} \in \ZZ[x] \quad \text{and} \quad g := 2^s (x + 1 / 2)^{\overline{\lfloor s / 2 \rfloor}} \in \ZZ[x].
\]
Then, $(f / g)|_{x = r} = p_{s, s}$.

As polynomials in $x$, the resultant of $f$ and $g$ has only prime factors at most $2 s - 1$. Since the resultant ${\rm res}_x(f, g)$ is a $\ZZ[x]$-linear combination of $f$ and $g$, the quotient ${\rm res}_x(f, g) / g$ is a $\ZZ[x]$-linear combination of the quotient $f / g$ and $1$. Speciallizing to $x = r$ implies that the quotient
\[
  q := \frac{{\rm res}_x(f, g)}{2^s (r + 1 / 2)^{\overline{\lfloor s / 2 \rfloor}}} = \frac{{\rm res}_x(f, g)}{2^s \lfloor s / 2 \rfloor! {r + \lfloor s / 2 \rfloor - 1 / 2 \choose \lfloor s / 2 \rfloor}} \in \QQ
\]
is a $\ZZ$-linear combination of $1$ and the specialization $(f / g)|_{x = r} = p_{s, s}$. Since $p_{s, s}$ is an integer, $q$ is also an integer.

Recall that the numerator of $q$, which is the resultant of $f$ and $g$, has only prime factors at most $2 s - 1$, so does the denominator of $q$. Thus,
\[
  {r + \lfloor s / 2 \rfloor - 1 / 2 \choose \lfloor s / 2 \rfloor}
\]
is an $(2 s - 1)$-smooth rational number. Since $2 s - 1 \leq 4 \lfloor s / 2 \rfloor + 1$, by \cref{prop:zinziberaceous}, either $s \leq 285$, or $r \leq s - 1$.

\noindent{\bf Case 1. $2 \leq s \leq 285$}

Since $p_{s, 2}$ is an integer, by \cref{prop:unensured}, $r \leq \frac{3}{4} s^2$. A computer search shows that for $2 \leq s \leq 285$ and $2 \leq r \leq \frac{3}{4} s^2$, the numbers $p_{s, 1}, \dots, p_{s, s}$ are never all integers at the same time.

\noindent{\bf Case 2. $2 \leq r \leq s - 1$}

We rewrite $p_{s, s}$ as follows:
\[
  p_{s, s} = \frac{(r + 1) \cdots (2 r - 1) \cdot r \cdots (r + s - 1)}{(r + s + 1) \ldots (2 r + s - 1)}.
\]
Since $p_{s, s}$ is an integer, neither of the numbers $r + s + 1, \ldots 2 r + s - 1$ could be a prime. Therefore, by \cref{prop:dominated}, $r + s \geq \rho_r$, hence $s \geq \rho_r - r$.

\noindent{\bf Case 2.1 $2 \leq r \leq 287$}

It suffices to prove that $s \geq 5000 r (14.5 + \ln r)^2$. This can be done by a computer search over $2 \leq r \leq 287$ and $s \leq 5000 r (14.5 + \ln r)^2$. The numbers $p_{s, 1}, \dots, p_{s, s}$ are never all integers at the same time.

\noindent{\bf Case 2.2 $r \geq 288$}

In this case, $s \geq \rho_r - r \geq 5000 r (14.6 + \ln r)^2 - r \geq 5000 r (14.5 + \ln r)^2$.
\end{proof}

\printbibliography

\end{document}